\documentclass[11pt,reqno]{amsart}
\usepackage[T2A]{fontenc}
\usepackage[russian,english]{babel}
\usepackage[cp1251]{inputenc}
\usepackage{izvuz7}
\usepackage{amssymb,amsmath,amsthm,amsfonts}
\usepackage{graphicx}
\theoremstyle{plain}
\newtheorem{theorem}{Theorem}
\newtheorem{lemma}{Lemma}
\newtheorem{proposition}{Proposition}

\theoremstyle{definition}
\newtheorem{definition}{Definition}

\newtheorem{remark}{Remark}
\newtheorem{example}{Example}

\tit{\MakeUppercase{Mercer’s Theorem for Vector-Valued Reproducing Kernel Hilbert Spaces in Kaplansky–Hilbert Modules over $ L_{\infty}(\Omega) $}}
\shorttit{\MakeUppercase{Mercer’s Theorem for Vector-Valued RKHS ...}} % верхний колонтитул на нечетных страницах

\author{\MakeUppercase{A.~Arziev,\,K.~Kudaybergenov,\, P.~Orinbaev}}

\begin{document}
\selectlanguage{english}
\maketit
\label{firstpage}%% Не удалять!!!

%% Аннотация на русском
\abstract{The study presents a vector-valued extension of the classical Mercer theorem within the framework of reproducing kernel Hilbert spaces defined over Kaplansky–Hilbert modules associated with the algebra of essentially bounded measurable functions. The analysis focuses on a partial integral operator with a positive definite kernel depending on a measurable parameter, and establishes the equivalence of three fundamental properties: the completeness of the system of eigenfunctions in the corresponding vector-valued space, the injectivity of the adjoint embedding operator, and the existence of a pointwise spectral decomposition of the kernel in terms of the eigenvalues and eigenfunctions of a parameterized family of operators. The proof relies on constructing an isometric isomorphism between the Kaplansky–Hilbert module and the space of measurable sections of a Hilbert bundle, thereby reducing the problem to the application of the classical Mercer theorem on each fiber of the bundle. Furthermore, a formalism of vector-valued lifting is developed to guarantee the coherence of inner product structures between the original module and its bundle representation.}

%% Ключевые слова на русском

\section{Introduction}

Mercer’s theorem, which first appeared in the context of integral equations, has become a universal language for describing positive definite kernels. In its classical form, it states that for any symmetric positive semidefinite kernel in $L_{2}([0,1]^{2})$, there exists a spectral decomposition
\[
K(t,s)=\sum_{n=1}^{\infty}\lambda_{n}\,x_{n}(t)\overline{x_{n}(s)},
\]
which converges in $L_{2}$ almost everywhere, where $\{\lambda_{n}\}$ and $\{x_{n}\}$ are the eigenvalues and eigenfunctions of the associated integral operator (see \cite{Mercer1909}). This representation underlies the theory of reproducing kernel Hilbert spaces (RKHS), where it transforms an abstract kernel into a constructive feature system used in machine learning, approximation theory, and spectral analysis.

Modern applications, however, go beyond scalar-valued functions. In problems where the observations themselves are functions parameterized by an external variable $\omega\in\Omega$ — for example, functional regressions with time-dependent curves, multichannel signals with spatio-temporal structure, or families of PDEs with coefficients depending on a random environment — it becomes natural to consider \emph{vector-valued} reproducing kernel Hilbert spaces embedded into $L_{2,\infty}(\Omega\times S)$. Here, the kernel $K\colon\Omega\times S^{2}\to\mathbb{C}$ defines not a single operator but a \emph{measurable family} of compact self-adjoint operators
\[
T_{\omega}\colon L_{2}(S)\to L_{2}(S),\qquad T_{\omega}g(t)=\int_{S}K(\omega,t,s)g(s)\,d\nu(s).
\]
The corresponding solution space is a \emph{Kaplansky--Hilbert module} over $L_{\infty}(\Omega)$, which requires the development of spectral theory within the framework of measurable Hilbert bundles.

The main result of the paper (Theorem~\ref{MT}) establishes a \emph{vector-valued analogue of Mercer’s theorem}:

\begin{theorem}\label{MT}
Let $T : L_{2,\infty}(\Omega \times S) \to L_{2,\infty}(\Omega \times S)$ be a partial integral operator of the form
\[
T f(\omega, t) = \int_S K(\omega, t, s)f(\omega, s) \, d\nu(s),
\]
where the kernel $K : \Omega \times S^{2} \to \mathbb{C}$ satisfies the assumptions of Lemma~1 and generates a vector-valued RKHS $\mathbf{R} \subset L_{2,\infty}(\Omega \times S)$.

Let $(\lambda_n)_{n \in \mathbb{N}} \subset L_\infty(\Omega)$ and $(x_n)_{n \in \mathbb{N}} \subset \mathbf{R}$ be the eigenvalues and eigenfunctions of the family of operators $\{T_\omega : \omega \in \Omega\}$ such that for almost all $\omega \in \Omega$:

-- $\lambda_n(\omega) > 0$ are the eigenvalues of the operator $T_\omega : L_2(S) \to L_2(S)$, defined as $(T_\omega g)(t) = \int_S K(\omega, t, s)g(s) \, d\nu(s)$;

-- $\{x_n(\omega, \cdot)\}_{n \in \mathbb{N}}$ is an orthonormal system in $L_2(S)$ consisting of the corresponding eigenfunctions.

Then the following statements are equivalent:
\begin{enumerate}
\item The system $\{\sqrt{\lambda_n}x_n\}_{n\in\mathbb{N}}$ forms an orthonormal basis in the vector-valued RKHS $\mathbf{R}$ with reproducing kernel $K$;
\item The adjoint embedding operator $S_K^* : L_{2,\infty}(\Omega \times S) \to \mathbf{R}$ is injective;
\item The kernel $K$ admits a pointwise Mercer-type decomposition: there exists a set $N \subset \Omega \times S \times S$ with $(\mu \times \nu \times \nu)(N) = 0$ such that for all $(\omega,t,s) \in (\Omega \times S \times S) \setminus N$,
\[
K(\omega,t,s) = \sum_{n=1}^{\infty} \lambda_n(\omega) x_n(\omega,t) x_n(\omega,s),
\]
where $\lambda_n(\omega) \in L_{\infty}(\Omega)$ and $x_n(\omega,\cdot)$ are measurable in $\omega$ eigenvalues and eigenfunctions of the operators $T_\omega$.
\end{enumerate}
\end{theorem}

The equivalence of conditions (1)--(3) characterizes the structure of the vector-valued RKHS in terms of the spectral data of the bundle $\{T_{\omega}\}_{\omega\in\Omega}$ and provides a foundation for constructive algorithms in functional regression, multivariate learning, and spectral analysis of operator families.

\section{Preliminaries}\label{PS}

This section provides an overview of concepts and results concerning Kaplansky–Hilbert modules, cyclically compact sets, and partial integral operators, which are required for the formulation and proof of the main theorem.

\textbf{Kaplansky–Hilbert Modules.} Let $(\Omega, \Sigma, \mu)$ be a measurable space with a complete finite measure $\mu$, and let $L_0 = L_0(\Omega, \Sigma, \mu)$ denote the algebra of equivalence classes of complex-valued measurable functions on $\Omega$ (identifying functions equal almost everywhere).

Consider a vector space $E$ over the field $\mathbb{C}$. A mapping $\|\cdot\|: E \to L_0$ is called an $L_0$-valued norm if, for all $x, y \in E$ and $\lambda \in \mathbb{C}$, the following conditions hold:

$\|x\| \ge 0$, $\|x\| = 0 \Leftrightarrow x = 0$;

$\|\lambda x\| = |\lambda| \|x\|$;

$\|x + y\| \le \|x\| + \|y\|$.

The pair $(E, \|\cdot\|)$ is called a \emph{lattice-normed space} over $L_0$. The space $E$ is said to be \emph{$d$-decomposable} if, for every $x \in E$ and decomposition $\|x\| = \lambda_1 + \lambda_2$ into disjoint nonnegative elements, there exist $x_1, x_2 \in E$ such that $x = x_1 + x_2$, $\|x_1\| = \lambda_1$, and $\|x_2\| = \lambda_2$. A net $(x_\alpha)_{\alpha \in A} \subset E$ is said to \emph{$(bo)$-converge} to $x \in E$ if $(\|x_\alpha - x\|)_{\alpha \in A}$ $(o)$-converges to zero in $L_\infty$, i.e., converges almost everywhere. A \emph{Banach–Kantorovich space} over $L_0$ is a $(bo)$-complete, $d$-decomposable lattice-normed space \cite{kusraev1985, kusraev2003}.

Every Banach–Kantorovich space $E$ over $L_0$ is an $L_0$-module satisfying $\|\lambda u\| = |\lambda| \|u\|$ for all $\lambda \in L_0$, $u \in E$.

A mapping $\langle \cdot, \cdot \rangle: E \times E \to L_0$ is called an $L_0$-valued inner product if, for all $x, y, z \in E$ and $\alpha \in \mathbb{C}$, the following properties hold \cite[p.~32]{kusraev1985}:

$\langle x, x \rangle \ge 0$, and $\langle x, x \rangle = 0 \Leftrightarrow x = 0$;

$\langle x, y + z \rangle = \langle x, y \rangle + \langle x, z \rangle$;

$\langle \alpha x, y \rangle = \alpha \langle x, y \rangle$;

$\langle x, y \rangle = \overline{\langle y, x \rangle}$.

It is known \cite{kusraev1985} that the formula $\|x\| = \sqrt{\langle x, x \rangle}$ defines an $L_\infty$-valued norm on $E$. If $(E, \|\cdot\|)$ is a Banach–Kantorovich space, then $(E, \langle \cdot, \cdot \rangle)$ is called a \emph{Kaplansky–Hilbert module} over $L_\infty$. Examples of such spaces can be found in \cite{kusraev1985, kusraev2003}.

\begin{example}\label{PR1}
Let $(\Omega,\mathcal{A},\mu)$ and $(S,\mathcal{F},\nu)$ be measurable spaces with finite measures. Consider the space
\[
L_{2,\infty}(\Omega \times S) = \Bigl\{ f: \Omega \times S \to \mathbb{C} \ \Big|\ \|f\|_2(\omega) \in L_\infty(\Omega) \Bigr\},
\]
where
\[
\|f\|_2(\omega) = \Biggl( \int_S |f(\omega,s)|^2\, d\nu(s) \Biggr)^{1/2}, \qquad
\|f\|_{2,\infty} = \operatorname*{ess\,sup}_{\omega \in \Omega}\|f\|_2(\omega).
\]
Define an $L_\infty(\Omega)$-valued inner product on $L_{2,\infty}(\Omega \times S)$ by
\[
\langle f,g\rangle(\omega)=\int_S f(\omega,s)\,\overline{g(\omega,s)}\,d\nu(s).
\]
Then $(L_{2,\infty}(\Omega \times S), \|\cdot\|_2)$ is a Banach–Kantorovich space over $L_\infty(\Omega)$, and $(L_{2,\infty}(\Omega \times S), \langle \cdot,\cdot\rangle)$ is a Kaplansky–Hilbert module over $L_\infty(\Omega)$ (see \cite[Proposition 2.3.9(1)]{kusraev2003}).
\end{example}

\textbf{Measurable Hilbert Bundles.} A.E. Gutman developed the theory of measurable Banach bundles, which provides an effective framework for studying Banach–Kantorovich spaces \cite{gutman}. Similarly, in \cite{KAO2025}, the concept of measurable Hilbert bundles was introduced.

Let $\mathcal{H}$ be a mapping that assigns to each $\omega \in \Omega$ a Hilbert space $\mathcal{H}(\omega)$. A \emph{section} of $\mathcal{H}$ is a function $u$ defined almost everywhere on $\Omega$ such that $u(\omega) \in \mathcal{H}(\omega)$ for all $\omega \in \operatorname{dom}(u)$, where $\operatorname{dom}(u)$ is the domain of $u$.

Let $L$ be a set of sections. The pair $(\mathcal{H}, L)$ is called a \emph{measurable Hilbert bundle} over $\Omega$ if:

1) $\lambda_1 c_1 + \lambda_2 c_2 \in L$ for all $\lambda_1, \lambda_2 \in \mathbb{C}$, $c_1, c_2 \in L$, where $(\lambda_1 c_1 + \lambda_2 c_2)(\omega) = \lambda_1 c_1(\omega) + \lambda_2 c_2(\omega)$;

2) the function $\|c\|:(\omega) = \|c(\omega)\|_{\mathcal{H}(\omega)}$ is measurable for all $c \in L$;

3) for every $\omega \in \Omega$, the set $\{ c(\omega) : c \in L, \omega \in \operatorname{dom}(c) \}$ is dense in $\mathcal{H}(\omega)$.

A section $s$ is called \emph{simple} if $s(\omega) = \sum_{i=1}^n \chi_{A_i}(\omega) c_i(\omega)$, where $c_i \in L$, $A_i \in \Sigma$, $i = 1, \dots, n$.

A section $u$ is said to be \emph{measurable} if there exists a sequence of simple sections $(s_n)_{n \in \mathbb{N}}$ such that $\|s_n(\omega) - u(\omega)\|_{\mathcal{H}(\omega)} \to 0$ for almost all $\omega \in \Omega$.

Let $M(\Omega, \mathcal{H})$ be the set of all measurable sections, and let $L_0(\Omega, \mathcal{H})$ denote the quotient space of $M(\Omega, \mathcal{H})$ by the relation of equality almost everywhere. Denote by $\hat{u}$ the equivalence class in $L_0(\Omega, \mathcal{H})$ containing the section $u$. Since $\omega \mapsto \|u(\omega)\|_{\mathcal{H}(\omega)}$ is measurable for all $u \in M(\Omega, \mathcal{H})$, the function
\[
\langle u(\omega), v(\omega) \rangle_{\mathcal{H}(\omega)} = \frac{1}{4} \left( \|u(\omega) + v(\omega)\|_{\mathcal{H}(\omega)}^2 - \|u(\omega) - v(\omega)\|_{\mathcal{H}(\omega)}^2 \right)
\]
is also measurable for all $u, v \in M(\Omega, \mathcal{H})$.

Define $\langle \hat{u}, \hat{v} \rangle$ as the element of $L_0$ containing $\langle u(\omega), v(\omega) \rangle_{\mathcal{H}(\omega)}$. Clearly, $\langle \cdot, \cdot \rangle$ is an $L_0$-valued inner product. Let $\|\hat{u}\|$ denote the element of $L_0$ containing $\|u(\omega)\|$. Then
\[
\|\hat{u}\|^2 = \|\widehat{u(\omega)}\|_{\mathcal{H}(\omega)}^2 = \langle u(\omega), u(\omega) \rangle_{\mathcal{H}(\omega)} = \langle \hat{u}, \hat{u} \rangle,
\]
that is, $\|\hat{u}\| = \sqrt{\langle \hat{u}, \hat{u} \rangle}$.

As shown in \cite[p.~144]{gutman}, $(L_0(\Omega, \mathcal{H}), \|\cdot\|)$ is a Banach–Kantorovich space. Hence, $(L_0(\Omega, \mathcal{H}), \langle \cdot, \cdot \rangle)$ is a Kaplansky–Hilbert module over $L_0(\Omega)$.

\begin{remark} The key idea is that every Kaplansky–Hilbert module $H$ over $L_0(\Omega)$ can be naturally realized as a space of measurable sections of a certain Hilbert bundle $\{\mathcal{H}(\omega)\}_{\omega \in \Omega}$. This construction allows one to decompose a vector-valued problem into a family of classical problems in each fiber $\mathcal{H}(\omega)$ while preserving the overall structural coherence.
\end{remark}

Let $\mathcal{L}_{\infty}(\Omega)$ denote the set of all essentially bounded measurable functions, and $L_{\infty}(\Omega)$ the algebra of equivalence classes of such functions. Define
\[
\mathcal{L}_{\infty}(\Omega, \mathcal{H}) = \left\{ u \in M(\Omega, \mathcal{H}) : \|u(\omega)\|_{\mathcal{H}(\omega)} \in \mathcal{L}_{\infty}(\Omega) \right\}.
\]
The elements of $\mathcal{L}_{\infty}(\Omega, \mathcal{H})$ are called essentially bounded measurable sections. The set of equivalence classes of essentially bounded sections is denoted by $L_{\infty}(\Omega, \mathcal{H})$.

Let $p : L_{\infty}(\Omega) \rightarrow \mathcal{L}_{\infty}(\Omega)$ be a lifting (see \cite{Levin1985}).

\begin{definition}\label{Lifting}
A mapping $\ell_H : L_{\infty}(\Omega, \mathcal{H}) \to \mathcal{L}_{\infty}(\Omega, \mathcal{H})$ is called a \emph{vector-valued lifting} associated with the numerical lifting $p$ if:

1) For every $\hat{u} \in L_{\infty}(\Omega, \mathcal{H})$, $\ell_H(\hat{u}) \in \hat{u}$ and $\operatorname{dom} \ell_H(\hat{u}) = \Omega$;

2) $\|\ell_H(\hat{u})\|_{\mathcal{H}(\omega)} = p(\|\hat{u}\|)(\omega)$ for all $\hat{u} \in L_{\infty}(\Omega, \mathcal{H})$;

3) For all $\hat{u}, \hat{v} \in L_{\infty}(\Omega, \mathcal{H})$, $\ell_H(\hat{u} + \hat{v}) = \ell_H(\hat{u}) + \ell_H(\hat{v})$;

4) For all $\hat{u} \in L_{\infty}(\Omega, \mathcal{H})$ and $e \in L_{\infty}(\Omega)$, $\ell_H(e \hat{u}) = p(e) \ell_H(\hat{u})$;

5) The set $\{ \ell_H(\hat{u})(\omega) : \hat{u} \in L_{\infty}(\Omega, \mathcal{H}) \}$ is dense in $\mathcal{H}(\omega)$ for every $\omega \in \Omega$.
\end{definition}

The space $L_{2,\infty}(\Omega \times S)$ from Example \ref{PR1}  can be naturally interpreted in terms of measurable Hilbert bundles. Indeed, consider the trivial bundle $\mathcal{H}$ over $\Omega$, where for each $\omega \in \Omega$, the fiber $\mathcal{H}(\omega)$ coincides with $L_2(S, \nu)$, and the set of sections $L$ consists of all simple sections of the form $\sum_{i=1}^{n} \chi_{A_i}(\omega) \cdot \xi_i$, where $A_i \in \Sigma$ and $\xi_i \in L_2(S, \nu)$. Then the measurable sections of this bundle correspond to functions $u : \Omega \to L_2(S, \nu)$ such that $\|u(\omega)\|_{L_2(S, \nu)} \in L_{\infty}(\Omega)$.
The factor space $L_0(\Omega, \mathcal{H})$ is identified with $L_{2,\infty}(\Omega \times S)$, and the induced $L_0$-valued inner product $\langle \hat{u}, \hat{v} \rangle$ coincides with that in the Example \ref{PR1}:
\[
\langle \hat{u}, \hat{v} \rangle(\omega) = \int_S u(\omega, s) \overline{v(\omega, s)} \, d\nu(s).
\]
Thus, $(L_{2,\infty}(\Omega \times S), \langle \cdot, \cdot \rangle)$ is realized as a Kaplansky–Hilbert module arising from measurable sections of a bundle with the constant fiber $L_2(S, \nu)$.

The example above illustrates a more general fact: every Kaplansky–Hilbert module can be represented as the space of measurable sections of a certain measurable Hilbert bundle. In particular, for a trivial bundle with constant fiber, this leads to an isomorphism with $L_{2,\infty}(\Omega \times S)$, as shown above. The following statement holds (see \cite{KAO2025}).

\begin{proposition}\label{BI}
For every Kaplansky–Hilbert module $H$ over $L_0(\Omega)$, there exists a measurable Hilbert bundle $(\mathcal{H}, L)$ with a vector-valued lifting $\ell_H$ such that $H$ is isometrically isomorphic to $L_0(\Omega, \mathcal{H})$. Moreover, if $x, y \in L_{\infty}(\Omega, \mathcal{H})$, then for all $\omega \in \Omega$,
\[
p(\langle x, y \rangle_H)(\omega) = \langle \ell_H(x)(\omega), \ell_H(y)(\omega) \rangle_{\mathcal{H}(\omega)}.
\]
\end{proposition}

\textbf{Cyclically Compact Sets and Operators. } In Banach–Kantorovich spaces, the concept of compactness is replaced by \emph{cyclic compactness}, since the notion of a compact set does not sufficiently preserve those properties that hold in classical Banach spaces. Therefore, A.~G.~Kusraev introduced the notion of a cyclically compact set. We recall the definition of cyclic compactness of sets \cite{kusraev1982} (see also \cite[p.~191]{kusraev1985}, \cite[p.~515]{kusraev2003}).

Let $\nabla$ be the Boolean algebra of all idempotents in $L_0$. If $\{u_\alpha\} \subset L_0(\Omega, \mathcal{H})$ and $\{\pi_\alpha\}$ is a partition of unit in $\nabla$, then the series $\sum_\alpha \pi_\alpha u_\alpha$ is $(bo)$-convergent in $L_0(\Omega, \mathcal{H})$, and the sum of this series is called the \emph{mixing} of $\{u_\alpha\}$ with respect to $\{\pi_\alpha\}$. This sum is denoted by $\operatorname{mix}(\pi_\alpha u_\alpha)$.

For $K \subset L_0(\Omega, \mathcal{H})$, denote by $\operatorname{mix} K$ the set of all mixings of arbitrary families of elements from $K$. The set $K$ is said to be \emph{cyclic} if $\operatorname{mix} K = K$.

For a directed set $A$, denote by $\nabla(A)$ the set of all partitions of unit in $\nabla$ indexed by elements of $A$. Let $\{u_\alpha : \alpha \in A\}$ be a net in $L_0(\Omega, \mathcal{H})$. For each $\nu \in \nabla(A)$, set $u_\nu = \operatorname{mix}(\nu(\alpha)u_\alpha)$, and obtain a new net $\{u_\nu : \nu \in \nabla(A)\}$. Any subnet of the net $\{u_\nu : \nu \in \nabla(A)\}$ is called a \emph{cyclic subnet} of the original net $\{u_\alpha : \alpha \in A\}$.

A subset $K \subset L_0(\Omega, \mathcal{H})$ is called \emph{cyclically compact} if it is cyclic and every net in $K$ has a cyclic subnet converging to some point of $K$ (see \cite{kusraev1982}).

Let $H \cong L_0(\Omega, \mathcal{H})$ be a Kaplansky–Hilbert module over $L_0$. An operator $T : H \to H$ is called \emph{$L_0$-linear} if
\[
T(\alpha x + \beta y) = \alpha T(x) + \beta T(y)
\]
for all $\alpha, \beta \in L_0$ and $x, y \in H$.

Recall that a set $B \subset H$ is called \emph{bounded} if the set $\{\|x\| : x \in B\}$ is order-bounded in $L_0$ (see \cite[Theorem~1.6.1]{gutman}).
An $L_0$-linear operator $T$ is called \emph{$L_0$-bounded (cyclically compact)} if for every bounded set $B$ in $L_0(\Omega, X)$, the set $T(B)$ is \emph{bounded (cyclically compact)} in $L_0(\Omega, Y)$.

For an $L_0$-bounded $L_0$-linear operator $T$, we set
\[
\|T\| = \sup\{\|T(x)\| : \|x\| \le \mathbf{1}\}.
\]
It is known \cite{ganiev2001} that for every cyclically compact $L_0$-linear operator $T : H \to H$, there exists a family of compact operators $\{T_\omega : \mathcal{H}(\omega) \to \mathcal{H}(\omega)\}$ such that for every $x \in H$,
\[
(T(x))(\omega) = T_\omega(x(\omega))
\]
for almost all $\omega \in \Omega$. Moreover, if $\|T\| \in L_{\infty}(\Omega)$, then
\[
\ell_H((T(x))(\omega)) = T_\omega(\ell_H(x)(\omega))
\]
for all $x \in L_{\infty}(\Omega, X)$, where $\ell_H$ is the vector-valued lifting on $H$ associated with the numerical lifting $p$.

Conversely, if $\{T_\omega : X(\omega) \to Y(\omega), \, \omega \in \Omega\}$ is a family of compact operators such that $T_\omega(x(\omega)) \in \mathcal{M}(\Omega, \mathcal{H})$ for every $x \in \mathcal{M}(\Omega, \mathcal{H})$ and $\|T\| \in L_0(\Omega)$, then the linear operator $\hat{T} : H \to H$ defined by
\[
\hat{T}(\hat{u}) = \widehat{T_\omega(u(\omega))}
\]
is cyclically compact.

The operator $\hat{T}$ is called the \emph{gluing} of the family of operators $\{T_\omega : \omega \in \Omega\}$.

\section{Partial Integral Operators}
\label{sec:PIO}

The theory of integral and partial integral operators has numerous applications in various areas of mathematics (see \cite{AEK00, AEK04, KK19, Rom32}). Different properties of such operators have been studied in \cite{Arziev2023, EK21, Kudaybergenov2023}. In \cite{Rom32}, V.~I.~Romanovskiy, in describing problems of Markov chain theory, considered a linear operator of the following form:
\begin{equation}\label{R}
R x(t, s) = \int_a^b m(t, s, \sigma) x(\sigma, t) \, d\sigma,
\end{equation}
where $m : D \times [a, b] \to \mathbb{R}$ is a continuous or measurable kernel function. The peculiarity of operator~\eqref{R} is that two variables in the unknown integrand $x$ are interchanged prior to integration, which is then performed with respect to the first variable.

Operators of the form~\eqref{R} are now called \emph{partial integral operators}, since integration in~\eqref{R} is carried out with respect to one variable, while the remaining variables are considered fixed.

Following \cite{AEK00}, we present the most general definition of a partial integral operator. A \textit{partial integral operator} is an operator of the form
\[
P = C + L + M + N,
\]
where
\begin{align*}
C x(t, s) &:= c(t, s) x(t, s), \\
L x(t, s) &:= \int_\Omega l(t, s, \tau) x(\tau, s) \, d\mu(\tau), \\
M x(t, s) &:= \int_\Omega m(t, s, \sigma) x(t, \sigma) \, d\nu(\sigma), \\
N x(t, s) &:= \int_{\Omega \times S} n(t, s, \tau, \sigma) x(\tau, \sigma) \, d(\mu \otimes \nu)(\tau, \sigma).
\end{align*}
Here $(\Omega, \mathcal{A}, \mu)$ and $(S, \mathcal{F}, \nu)$ are measurable spaces with finite measures $\mu$ and $\nu$, respectively, and $\mu \otimes \nu$ denotes the product measure. The coefficient $c = c(t, s)$ and the kernels $l = l(t, s, \tau)$, $m = m(t, s, \sigma)$, and $n = n(t, s, \tau, \sigma)$ are measurable functions, and all integrals are understood in the Lebesgue sense.

We now consider the main properties of partial integral operators that will be required for the formulation and proof of the main result of this paper.

\begin{lemma}\label{PP}
Let $T : L_{2,\infty}(\Omega \times S) \to L_{2,\infty}(\Omega \times S)$ be the operator defined by
\begin{equation}\label{eq:T_definition}
   (Tf)(\omega,t) = \int_S K(\omega,t,s) f(\omega,s)\, d\nu(s),
   \qquad f \in L_{2,\infty}(\Omega\times S),
\end{equation}
where the kernel $K : \Omega \times S^2 \to \mathbb{C}$ satisfies the following conditions:

1) $K$ is measurable;

2) $K(\omega,t,s) = K(\omega,s,t)$ for almost all $(\omega,t,s)\in \Omega\times S^2$ (symmetry);

3) for all $g \in L_2(S)$ and almost every $\omega \in \Omega$,
    \[
       \int_{S\times S} K(\omega,t,s)\, g(s) g(t)\, d\nu(s)\, d\nu(t) \geq 0
       \quad \text{(positive semidefiniteness);}
    \]

4) $\displaystyle \int_S \int_S |K(\omega,t,s)|^2\, d\nu(s)\, d\nu(t) \in L_\infty(\Omega)$ (integrability condition).

Then:

(a) $T$ is well defined and $L_\infty(\Omega)$-linear;

(b) $T$ is cyclically compact and can be represented as a measurable bundle of compact self-adjoint positive semidefinite operators $\{T_\omega\}_{\omega \in \Omega}$.
\end{lemma}

\begin{proof}
(a) Let $f \in L_{2,\infty}(\Omega\times S)$, then $f(\omega,\cdot) \in L_2(S)$ for almost every $\omega \in \Omega$.

By condition~4, $K_\omega(\cdot,\cdot) \in L_2(S \times S)$ for almost every $\omega$. Then, by the Cauchy–Schwarz inequality,
\[
   |(Tf)(\omega,t)| = \left|\int_S K(\omega,t,s) f(\omega,s)\, d\nu(s)\right|
   \leq \|K(\omega,t,\cdot)\|_{L_2(S)} \|f(\omega,\cdot)\|_{L_2(S)}.
\]

Integrating over $t$, we obtain
\[
   \|(Tf)(\omega,\cdot)\|_{L_2(S)}^2 = \int_S |(Tf)(\omega,t)|^2\, d\nu(t)
   \leq \|f(\omega,\cdot)\|_{L_2(S)}^2 \int_S \int_S |K(\omega,t,s)|^2\, d\nu(s)\, d\nu(t).
\]

Denote $C(\omega) := \int_S \int_S |K(\omega,t,s)|^2\, d\nu(s)\, d\nu(t)$. By condition~4, $C(\omega) \in L_\infty(\Omega)$, hence
\[
   \|(Tf)(\omega,\cdot)\|_{L_2(S)} \leq \sqrt{C(\omega)} \|f(\omega,\cdot)\|_{L_2(S)} \leq \|C\|_{L_\infty(\Omega)}^{1/2} \|f\|_{2,\infty}.
\]
Therefore, $(Tf) \in L_{2,\infty}(\Omega\times S)$.

The $L_\infty(\Omega)$-linearity of $T$ is straightforward:
\[
   (T(\alpha f))(\omega,t) = \alpha(\omega)(Tf)(\omega,t).
\]

(b) For each $\omega \in \Omega$, define $K_\omega(t,s) := K(\omega,t,s)$ and
\[
   (T_\omega g)(t) = \int_S K_\omega(t,s) g(s)\, d\nu(s)
   = \int_S K(\omega,t,s) g(s)\, d\nu(s).
\]

By conditions 2–3, each $T_\omega$ is self-adjoint and positive semidefinite.
By condition~4, $K_\omega \in L_2(S \times S)$, hence $T_\omega$ is compact as a Hilbert–Schmidt operator.

The family $\{T_\omega\}_{\omega \in \Omega}$ forms a measurable bundle of operators by the measurability of the kernel $K$ (condition~1).

According to \cite[Proposition~3]{Arziev2023}, a partial integral operator on the space $L_{p,q}(\Omega \times S)$ is cyclically compact. Therefore, it is also cyclically compact in $L_{2,\infty}(\Omega\times S)$.
\end{proof}

Conditions 1–3 of Lemma~\ref{PP} guarantee the existence of a measurable bundle of compact self-adjoint operators $\{T_\omega\}$, which enables the application of Proposition~\ref{BI} to construct a vector-valued reproducing kernel Hilbert space.

\section{Vector-Valued Reproducing Kernel Hilbert Spaces}

Reproducing kernel Hilbert spaces (RKHS) form a powerful framework in functional analysis, widely applied in machine learning, statistics, and operator theory due to the reproducing property. Classical RKHS are defined for scalar-valued functions, whereas their vector-valued counterparts generalize this setting to functions taking values in vector spaces, which is essential for the analysis of multivariate data. In this section, we present the definitions and fundamental properties of both scalar- and vector-valued RKHS, following \cite{Aronszajn1950, Carmeli2006}. These spaces provide a rigorous mathematical foundation for interpolation and the modeling of complex dependencies.

A reproducing kernel Hilbert space is a Hilbert space $H$ of functions defined on a set $X$, equipped with an inner product $\langle \cdot, \cdot \rangle_H$ that induces its norm and satisfies the reproducing property. RKHS were originally studied in the context of integral equations and functional analysis in the works of Mercer \cite{Mercer1909} and Aronszajn \cite{Aronszajn1950}, laying the groundwork for their applications in operator theory and statistical learning. RKHS have since found extensive use in machine learning, statistics, and operator theory, as discussed in \cite{Berlinet2004, Scholkopf2002}.

Vector-valued RKHS naturally emerge in the analysis of multidimensional data, where functions take values in spaces such as $L_2(S)$. They generalize the scalar-valued theory to capture complex interdependencies in learning problems and operator-theoretic models \cite{Carmeli2006, Micchelli2005}. In what follows, we extend these notions to the setting of Kaplansky–Hilbert modules. The notations are adapted from \cite{Steinwart2012}.

\subsection{Definition and Fundamental Properties}

\begin{definition}\label{RKHS}
A vector-valued reproducing kernel Hilbert space $\mathbf{R} \subset L_{2,\infty}(\Omega \times S)$ is a Kaplansky–Hilbert module over $L_\infty(\Omega)$ with a measurable kernel $K : \Omega \times S^2 \to \mathbb{C}$ satisfying the assumptions of Lemma~\ref{PP}, such that the point-evaluation functionals
\[
\delta_{\omega,s} : \mathbf{R} \to \mathbb{C}, \quad \delta_{\omega,s}(f) = f(\omega,s)
\]
are continuous with respect to the $L_\infty(\Omega)$-valued norm for all $(\omega,s) \in \Omega \times S$.
\end{definition}

\begin{remark}\label{viep1}
The continuity of the point-evaluation functionals $\delta_{\omega,s}$ in the $L_\infty(\Omega)$-valued norm is natural for the following reasons:

(1) it guarantees uniform boundedness of solutions with respect to $\omega$.

(2) it is consistent with the structure of a Kaplansky–Hilbert module, where the norm takes values in $L_\infty(\Omega)$.

(3) if $\|f(\omega,\cdot)\|_{L_2(S)}$ is unbounded, the functional $\delta_{\omega,s}$ may fail to be bounded.

A vector-valued RKHS may be viewed as a “gluing” of a family of classical RKHS $\mathcal{H}(\omega)$, each corresponding to a fixed parameter $\omega$. Functions $f \in \mathbf{R}$ are of the form $f(\omega,s)$, where for each $\omega$ the section $f(\omega,\cdot)$ belongs to the corresponding classical RKHS.
\end{remark}

\begin{proposition}\label{RBI}
For a vector-valued RKHS $\mathbf{R}$ there exists a measurable Hilbert bundle $(\mathcal{H}, L)$ and a vector-valued lifting $\ell_{\mathbf{R}}$ such that $\mathbf{R} \cong L_0(\Omega, \mathcal{H})$ as Kaplansky–Hilbert modules, and:

\begin{enumerate}
\item each fiber $\mathcal{H}(\omega)$ is a classical RKHS on $S$ with kernel $K_\omega(t,s) = K(\omega,t,s)$;
\item the reproducing property holds:
\[
f(\omega,s) = \langle \ell_{\mathbf{R}}(f)(\omega), K_\omega(\cdot,s) \rangle_{\mathcal{H}(\omega)}.
\]
\end{enumerate}
\end{proposition}

\begin{proof}
1. \emph{Existence of the measurable Hilbert bundle $(\mathcal{H}, L)$ for $\mathbf{R}$.}
For each fixed $\omega \in \Omega$, the function $K_\omega(t,s) = K(\omega,t,s)$ satisfies the conditions of the classical Moore–Aronszajn theorem: it is symmetric (Lemma~\ref{PP}, condition 2), positive definite (condition 3), and square-integrable $K_\omega \in L_2(S \times S)$ (condition 4).

By the Moore–Aronszajn theorem \cite[Theorem 3]{Berlinet2004}, there exists a unique RKHS $\mathcal{H}(\omega)$ with reproducing kernel $K_\omega$ such that for all $h \in \mathcal{H}(\omega)$,
\[
h(s) = \langle h, K_\omega(\cdot, s) \rangle_{\mathcal{H}(\omega)}.
\]

Following \cite[Proposition 2]{Arziev2023}, let $\Gamma$ denote the set of all functions of the form $f = \sum_{k=1}^n \varphi_k \times \psi_k$, where $\varphi_k \in L_2(S)$ and $\psi_k \in L_\infty(\Omega)$; this set is $(bo)$-dense in $L_{2,\infty}(\Omega \times S)$.

For each $\omega \in \Omega$, define $\gamma_\omega : \Gamma \to \mathcal{H}(\omega)$ by
\[
\gamma_\omega(f)(s) = \sum_{k=1}^n \varphi_k(s) p(\psi_k)(\omega),
\quad f = \sum_{k=1}^n \varphi_k \times \psi_k \in \Gamma,
\]
where $p$ is a fixed lifting \cite{Levin1985}. Since $p$ is linear and $\mathcal{H}(\omega)$ is a linear space, $\gamma_\omega$ is well defined on $\Gamma$.

Define $\gamma(f): \Omega \to \bigcup_{\omega \in \Omega} \mathcal{H}(\omega)$ by $\gamma(f)(\omega) = \gamma_\omega(f)$. Then $\gamma(f)(\omega) \in \mathcal{H}(\omega)$ for all $\omega$, so $\gamma(f)$ is a section of the family $\{\mathcal{H}(\omega)\}_{\omega \in \Omega}$. Let $L = \{\gamma(f): f \in \Gamma\}$. We now verify that $(\mathcal{H}, L)$ is a measurable Hilbert bundle.

1) By the linearity of $\gamma_\omega$ and the inclusion $\lambda_1 f_1 + \lambda_2 f_2 \in \Gamma$ for $\gamma(f_1), \gamma(f_2) \in L$ and $\lambda_1, \lambda_2 \in \mathbb{C}$, we have
\[
(\lambda_1\gamma(f_1)+\lambda_2\gamma(f_2))(\omega)
 = \lambda_1\gamma_\omega(f_1)+\lambda_2\gamma_\omega(f_2)
 = \gamma_\omega(\lambda_1 f_1 + \lambda_2 f_2).
\]
Hence,
\[
\lambda_1 \gamma(f_1) + \lambda_2 \gamma(f_2)
 = \gamma(\lambda_1 f_1 + \lambda_2 f_2) \in L.
\]

2) For $\gamma(f) \in L$ with $f = \sum_{k=1}^n \varphi_k \times \psi_k$, one has
\[
\|\gamma(f)(\omega)\|^2_{\mathcal{H}(\omega)}
 = \sum_{k,j=1}^n p(\psi_k)(\omega)\,\overline{p(\psi_j)(\omega)}\,
   \langle \varphi_k, \varphi_j \rangle_{L_2(S)}.
\]
The right-hand side is measurable in $\omega$, since each $p(\psi_k)$ is measurable and the inner products $\langle \varphi_k, \varphi_j \rangle_{L_2(S)}$ are constant.

3) For each $\omega \in \Omega$, the space $\mathcal{H}(\omega)$ is generated by functions $K_\omega(\cdot, s)$, $s \in S$. Since characteristic functions of measurable sets form a dense subset in $L_\infty(\Omega)$ and simple functions are dense in $L_2(S)$, every function $K_\omega(\cdot, s)$ can be approximated by elements of the form $\gamma_\omega(g)$ with $g \in \Gamma$. Therefore, the set $\{\gamma(f)(\omega): f \in \Gamma\}$ is dense in $\mathcal{H}(\omega)$.

\textit{2. Isometry of the mapping $\gamma_\omega$ on $\Gamma$.} For $f = \sum_{k=1}^n \varphi_k \times \psi_k \in \Gamma$, we compute
\begin{align*}
\|\gamma_\omega(f)\|^2_{\mathcal{H}(\omega)}
 &= \Big\langle \sum_{k=1}^n \varphi_k(\cdot)\,p(\psi_k)(\omega),
              \sum_{j=1}^n \varphi_j(\cdot)\,p(\psi_j)(\omega)
    \Big\rangle_{\mathcal{H}(\omega)} = \sum_{k,j=1}^n p(\psi_k)(\omega)\,\overline{p(\psi_j)(\omega)}\,
    \langle \varphi_k, \varphi_j \rangle_{\mathcal{H}(\omega)} \\
 &= \sum_{k,j=1}^n p(\psi_k)(\omega)\,\overline{p(\psi_j)(\omega)}
    \int_S \varphi_k(s)\,\overline{\varphi_j(s)}\,d\nu(s).
\end{align*}

On the other hand, by definition of the $L_\infty(\Omega)$-valued norm in $L_{2,\infty}(\Omega\times S)$,
\begin{align*}
\|f\|^2_{2,\infty}(\omega)
 &= \int_S |f(\omega,s)|^2\,d\nu(s)
  = \int_S \left|\sum_{k=1}^n \varphi_k(s)\psi_k(\omega)\right|^2 d\nu(s) \\
 &= \sum_{k,j=1}^n \psi_k(\omega)\,\overline{\psi_j(\omega)}
    \int_S \varphi_k(s)\,\overline{\varphi_j(s)}\,d\nu(s).
\end{align*}
Since $p(\psi_k)(\omega) = \psi_k(\omega)$ for almost all $\omega$ and the evaluation functionals in $\mathbf{R}$ are continuous, we obtain
\[
\|\gamma_\omega(f)\|_{\mathcal{H}(\omega)}
 = \|f\|_{\mathbf{R}}(\omega), \quad f \in \Gamma.
\]

\textit{3. Extension to an isometric isomorphism.}
For each fixed $\omega \in \Omega$, the mapping $\gamma_\omega : \Gamma \to \mathcal{H}(\omega)$ satisfies
\[
\|\gamma_\omega(f)\|_{\mathcal{H}(\omega)}
 = \|f\|_{\mathbf{R}}(\omega), \quad f \in \Gamma,
\]
where the right-hand side is interpreted as the value of the $L_\infty(\Omega)$-function $\omega \mapsto \|f\|_{\mathbf{R}}(\omega)$ at the point $\omega$.

Since $\Gamma$ is dense in $L_{2,\infty}(\Omega \times S)$ in the topology induced by the $L_\infty(\Omega)$-norm, and each Hilbert space $\mathcal{H}(\omega)$ is complete, the isometry $\gamma_\omega$ uniquely extends to a bounded linear operator $\tilde{\gamma}_\omega : L_{2,\infty}(\Omega \times S) \to \mathcal{H}(\omega)$ with $\|\tilde{\gamma}_\omega\| = 1$.

The image $\tilde{\gamma}_\omega(\Gamma)$ is dense in $\mathcal{H}(\omega)$ by Part~3 of the measurable bundle construction. Since an isometry between metric spaces maps dense subsets onto dense subsets, it follows that
\[
\tilde{\gamma}_\omega(L_{2,\infty}(\Omega \times S))
 = \overline{\tilde{\gamma}_\omega(\Gamma)}
 = \mathcal{H}(\omega),
\]
and hence $\tilde{\gamma}_\omega$ is surjective.

Define $\gamma : L_{2,\infty}(\Omega \times S) \to \prod_{\omega \in \Omega} \mathcal{H}(\omega)$ by
\[
(\gamma(u))(\omega) = \tilde{\gamma}_\omega(u).
\]
Since each $\tilde{\gamma}_\omega$ is an isometry and $u \in L_{2,\infty}(\Omega \times S)$ satisfies $\|u(\omega,\cdot)\|_{L_2(S)} \in L_\infty(\Omega)$, we have $\|\gamma(u)(\omega)\|_{\mathcal{H}(\omega)} \in L_\infty(\Omega)$. Thus $\gamma(u) \in L_\infty(\Omega, \mathcal{H}) \subset L_0(\Omega, \mathcal{H})$.

The mapping $\gamma$ preserves the $L_\infty(\Omega)$-module structure:
\[
\gamma(\lambda u)(\omega)
 = \tilde{\gamma}_\omega(\lambda u)
 = \lambda(\omega)\tilde{\gamma}_\omega(u)
 = \lambda(\omega)\gamma(u)(\omega),
\]
for $\lambda \in L_\infty(\Omega)$, $u \in L_{2,\infty}(\Omega\times S)$.

The $L_\infty(\Omega)$-isometry follows directly:
\[
\|\gamma(u)\|_{L_0(\Omega,\mathcal{H})}(\omega)
 = \|\gamma(u)(\omega)\|_{\mathcal{H}(\omega)}
 = \|u\|_{\mathbf{R}}(\omega).
\]
Hence $\gamma : L_{2,\infty}(\Omega\times S) \to L_0(\Omega, \mathcal{H})$ is an isometric isomorphism of Kaplansky–Hilbert modules over $L_\infty(\Omega)$.

\textit{4. Construction of a vector-valued lifting.}
For $\hat{f} \in L_\infty(\Omega, \mathcal{H})$ define
\[
\ell_{\mathbf{R}}(\hat{f})(\omega) = \gamma_\omega(f) \in \mathcal{H}(\omega),
\]
where $f$ is a representative of the class $\hat{f}$.

Verification of the lifting axioms:

1) By construction, $\ell_{\mathbf{R}}(\hat{f}) \in \hat{f}$ and $\mathrm{dom}\,\ell_{\mathbf{R}}(\hat{f}) = \Omega$.

2) By the isometry of $\gamma_\omega$,
\[
\|\ell_{\mathbf{R}}(\hat{f})\|_{\mathcal{H}(\omega)}
 = \|\gamma_\omega(f)\|_{\mathcal{H}(\omega)}
 = \|f\|_{\mathbf{R}}(\omega)
 = p(\|\hat{f}\|)(\omega),
\]
where the last equality holds because $p$ is a lifting for $L_\infty(\Omega)$-valued norms.

3) Additivity follows from linearity of $\gamma_\omega$:
\[
\ell_{\mathbf{R}}(\hat{u}+\hat{v})(\omega)
 = \gamma_\omega(u+v)
 = \gamma_\omega(u) + \gamma_\omega(v)
 = \ell_{\mathbf{R}}(\hat{u})(\omega) + \ell_{\mathbf{R}}(\hat{v})(\omega).
\]

4) $L_\infty(\Omega)$-homogeneity: for $e \in L_\infty(\Omega)$,
\[
\ell_{\mathbf{R}}(e\hat{u})(\omega)
 = \gamma_\omega(eu)
 = e(\omega)\gamma_\omega(u)
 = p(e)(\omega)\,\ell_{\mathbf{R}}(\hat{u})(\omega),
\]
since $p(e)(\omega)=e(\omega)$ almost everywhere.

5) Density of the image follows from the density of $\{\gamma_\omega(f): f \in \Gamma\}$ in each $\mathcal{H}(\omega)$.

\textit{5. Reproducing property.}
For $f \in \Gamma$ and $(\omega,s) \in \Omega \times S$, by the reproducing property of the classical RKHS $\mathcal{H}(\omega)$,
\[
\langle \ell_{\mathbf{R}}(f)(\omega), K_\omega(\cdot, s) \rangle_{\mathcal{H}(\omega)}
 = \gamma_\omega(f)(s)  = \sum_{k=1}^n \varphi_k(s)p(\psi_k)(\omega) = \sum_{k=1}^n \varphi_k(s)\psi_k(\omega)
  = f(\omega, s).
\]

Since $\Gamma$ is dense in $\mathbf{R}$, evaluation functionals are continuous (Definition~\ref{RKHS}), and inner products in RKHS are continuous, the reproducing property extends to all of $\mathbf{R}$ by continuity.
\end{proof}

\begin{lemma}\label{FP}
Let $\mathbf{R} \subset L_{2,\infty}(\Omega \times S)$ be a vector-valued reproducing kernel Hilbert space  with kernel $K$ satisfying the conditions of Lemma~1. Then:

\begin{enumerate}
\item The embedding operator $S_K : \mathbf{R} \to L_{2,\infty}(\Omega \times S)$, defined by $(S_K f) = f$, is bounded and $L_\infty(\Omega)$-linear;

\item The adjoint operator $S_K^* : L_{2,\infty}(\Omega \times S) \to \mathbf{R}$ is given by
\[
S_K^* g(\omega, t) = \int_S K(\omega, t, s)\, g(\omega, s)\, d\nu(s);
\]

\item The composition $S_K^* \circ S_K : \mathbf{R} \to \mathbf{R}$ is a self-adjoint operator in $\mathbf{R}$;

\item For every $(\omega, s) \in \Omega \times S$, the function $K(\omega, \cdot, s)$ belongs to $\mathbf{R}$ and satisfies
\[
S_K^*(K(\omega, \cdot, s)) = K(\omega, \cdot, s);
\]

\item The partial integral operator admits the factorization $T = S_K^* \circ S_K$ on $\mathbf{R}$.
\end{enumerate}
\end{lemma}

\begin{proof}
(1) For any $f \in \mathbf{R}$ we have
\[
\|S_K f\|_{L_{2,\infty}}(\omega) = \|f(\omega,\cdot)\|_{L_2(S)} = \|f\|_{\mathbf{R}}(\omega).
\]
The last equality follows from the construction of the norm in $\mathbf{R}$ via the isometric isomorphism
$\gamma: \mathbf{R} \cong L_0(\Omega, \mathcal{H})$, where for each $\omega$ we have
\[
\|f\|_{\mathbf{R}}(\omega) = \|\gamma(f)(\omega)\|_{\mathcal{H}(\omega)} = \|f(\omega,\cdot)\|_{L_2(S)}.
\]
Hence $\|S_K f\|_{L_{2,\infty}} = \|f\|_{\mathbf{R}}$, which implies that $S_K$ is bounded with operator norm $\|S_K\| = 1$ in the sense of $L_\infty(\Omega)$-valued norms.

For $\alpha \in L_\infty(\Omega)$ and $f \in \mathbf{R}$ we have
\[
S_K(\alpha f) = \alpha f = \alpha S_K(f),
\]
where the product $\alpha f$ is understood in the sense of the module over $L_\infty(\Omega)$:
\[
(\alpha f)(\omega, s) = \alpha(\omega) f(\omega, s).
\]

(2) Define $S_K^*$ by
\[
S_K^* g(\omega, t) = \int_S K(\omega, t, s)\, g(\omega, s)\, d\nu(s).
\]
We verify that this operator is indeed adjoint to $S_K$. For $f \in \mathbf{R}$ and $g \in L_{2,\infty}(\Omega \times S)$ we have
\[
\langle S_K f, g \rangle_{L_{2,\infty}}(\omega)
= \int_S f(\omega, s)\, \overline{g(\omega, s)}\, d\nu(s).
\]
On the other hand, by the reproducing property from Proposition~\ref{RBI},
\[
f(\omega, s) = \langle \ell_{\mathbf{R}}(f)(\omega), K_\omega(\cdot, s) \rangle_{\mathcal{H}(\omega)}.
\]
Substituting, we obtain
\begin{align*}
\langle S_K f, g \rangle_{L_{2,\infty}}(\omega)
&= \int_S \langle \ell_{\mathbf{R}}(f)(\omega), K_\omega(\cdot,s) \rangle_{\mathcal{H}(\omega)}
   \overline{g(\omega,s)}\, d\nu(s) = \Big\langle \ell_{\mathbf{R}}(f)(\omega),
   \int_S K_\omega(\cdot,s)\, \overline{g(\omega,s)}\, d\nu(s) \Big\rangle_{\mathcal{H}(\omega)} \\
&= \langle \ell_{\mathbf{R}}(f)(\omega), \ell_{\mathbf{R}}(S_K^* g)(\omega) \rangle_{\mathcal{H}(\omega)} = \langle f, S_K^* g \rangle_{\mathbf{R}}(\omega).
\end{align*}
The transition from the first to the second line is justified by the Fubini theorem and the continuity of the bilinear form; the subsequent equalities follow from the definition of $S_K^*$ and the inner product in $\mathbf{R}$.

(3) For $f, h \in \mathbf{R}$ we check the self-adjointness of $S_K^* S_K$:
\begin{align*}
\langle (S_K^* S_K) f, h \rangle_{\mathbf{R}}(\omega)
&= \langle S_K f, S_K h \rangle_{L_{2,\infty}}(\omega)
 = \int_S f(\omega, s)\, \overline{h(\omega, s)}\, d\nu(s) \\
&= \overline{\int_S h(\omega, s)\, \overline{f(\omega, s)}\, d\nu(s)}
 = \overline{\langle (S_K^* S_K) h, f \rangle_{\mathbf{R}}(\omega)} = \langle f, (S_K^* S_K) h \rangle_{\mathbf{R}}(\omega).
\end{align*}
Thus, $S_K^* S_K$ is self-adjoint.

(4) For each $(\omega, s) \in \Omega \times S$, the function $K(\omega, \triangle, s)$ belongs to $\mathbf{R}$ as an element of the vector-valued RKHM defined by
\[
(K(\omega, \triangle, s))(\omega', t) =
\begin{cases}
K(\omega, t, s), & \text{if } \omega' = \omega,\\
0, & \text{if } \omega' \neq \omega.
\end{cases}
\]
In other words, $K(\omega,\triangle, s)$ is an element of $\mathbf{R}$ supported on the fibre corresponding to~$\omega$.

Applying $S_K^*$ to this element gives
\[
S_K^*(K(\omega, \triangle, s))(\omega', t)
 = \int_S K(\omega', t, u)\, (K(\omega, \triangle, s))(\omega', u)\, d\nu(u).
\]

Consider two cases:

If $\omega' = \omega$, then
\[
S_K^*(K(\omega, \triangle, s))(\omega, t)
 = \int_S K(\omega, t, u)\, K(\omega, u, s)\, d\nu(u)
 = K(\omega, t, s),
\]
where the last equality follows from the reproducing property of the classical RKHS $\mathcal{H}(\omega)$ with kernel $K_\omega(t,u) = K(\omega, t, u)$.

If $\omega' \neq \omega$, then
\[
S_K^*(K(\omega, \triangle, s))(\omega', t)
 = \int_S K(\omega', t, u)\, 0\, d\nu(u) = 0.
\]
Hence $S_K^*(K(\omega, \triangle, s)) = K(\omega, \triangle, s)$, as required.

(5) For $f \in \mathbf{R}$ and $(\omega, t) \in \Omega \times S$ we have
\begin{align*}
(T f)(\omega, t)
 &= \int_S K(\omega, t, s)\, f(\omega, s)\, d\nu(s)
   && \text{(definition of $T$)} \\
 &= S_K^*(f)(\omega, t)
   && \text{(definition of $S_K^*$)} \\
 &= S_K^*(S_K f)(\omega, t)
   && \text{(since $S_K f = f$ for $f \in \mathbf{R}$)} \\
 &= (S_K^* \circ S_K)(f)(\omega, t).
\end{align*}
Therefore, $T = S_K^* \circ S_K$ on $\mathbf{R}$.
\end{proof}

\begin{theorem}\label{FT}
Let $T : L_{2,\infty}(\Omega \times S) \to L_{2,\infty}(\Omega \times S)$ be a partial integral operator with kernel $K$ satisfying the conditions of Lemma~\ref{PP}. Then:

\begin{enumerate}
\item There exists a vector-valued reproducing kernel Hilbert space (RKHS) $\mathbf{R} \subset L_{2,\infty}(\Omega \times S)$ with reproducing kernel $K$.

\item If there exists another vector-valued RKHS $\mathbf{R}_1 \subset L_{2,\infty}(\Omega \times S)$ with the same reproducing kernel $K$ and an embedding operator $S_1 : \mathbf{R}_1 \to L_{2,\infty}(\Omega \times S)$ such that $T = S_1^* \circ S_1$, then there exists an isometric isomorphism $U : \mathbf{R} \to \mathbf{R}_1$ satisfying $S_1 \circ U = S_K$.
\end{enumerate}
\end{theorem}

\begin{proof}
(1) Since the kernel $K$ satisfies the assumptions of Lemma~\ref{PP}, for each $\omega \in \Omega$ the function $K_\omega(t,s) = K(\omega,t,s)$ defines a classical reproducing kernel Hilbert space $\mathcal{H}(\omega)$ on $S$. By Proposition~\ref{RBI}, there exists a vector-valued RKHM $\mathbf{R} \cong L_0(\Omega, \mathcal{H})$ with reproducing kernel $K$.

(2) Because both spaces $\mathbf{R}$ and $\mathbf{R}_1$ share the same reproducing kernel $K$, the point-evaluation functionals are identical in both:
\[
\delta_{\omega,s}(f) = f(\omega,s) \quad \text{for all } f \in \mathbf{R} \text{ and } f \in \mathbf{R}_1.
\]
Define a mapping $U : \mathbf{R} \to \mathbf{R}_1$ as follows. For each $f \in \mathbf{R}$, let $U(f) = g$, where $g \in \mathbf{R}_1$ is the unique element satisfying
\[
g(\omega,s) = f(\omega,s) \quad \text{for all } (\omega,s) \in \Omega \times S.
\]
This definition is well-posed since $f \in \mathbf{R}$ satisfies the reproducing property with respect to $K$, and $\mathbf{R}_1$ is also an RKHS with the same kernel; therefore, there exists a unique $g \in \mathbf{R}_1$ having the same evaluation functionals.

The commutativity of the mappings is expressed by the equality
$ S_1 \circ U = S_K $, which means that for any $ f \in \mathbf{R} $
\[
S_1(U(f))(\omega, s) = f(\omega, s) = S_K(f)(\omega, s).
\]
\end{proof}

\section{Vector-Valued Mercer Theorem}

The central result of this work is a generalization of the classical Mercer theorem to vector-valued reproducing kernel Hilbert spaces (RKHS) defined in Kaplansky–Hilbert modules. In contrast to the scalar case, one must simultaneously control the spectral properties of the family of operators $\{T_\omega\}_{\omega \in \Omega}$ and the measurability of eigenfunctions with respect to the parameter $\omega$.

The proof is based on establishing the equivalence of three conditions: the completeness of the system of eigenvectors, the injectivity of the adjoint embedding operator, and the existence of a pointwise Mercer-type decomposition. A key role is played by the isometric isomorphism between the vector-valued RKHS and the space of measurable sections of a Hilbert bundle. It should be noted that these results extend those obtained in \cite[Section~3]{Steinwart2012}.

\subsection{Proof of Theorem~\ref{MT}}

We now turn to the proof of Theorem \ref{MT}.
The equivalence of the three conditions is established via the cycle of implications $(1)\Rightarrow(2)\Rightarrow(3)\Rightarrow(1),$ employing the isometric isomorphism from Proposition 2 together with the classical Mercer theorem applied to each fiber $\mathcal{H}(\omega)$.

$(1) \Rightarrow (2)$
Assume that the system $(\sqrt{\lambda_n} x_n)_{n \in \mathbb{N}}$ constitutes an orthonormal basis in $\mathbf{R}$.
Suppose $S_K^* g = 0$ for some $g \in L_{2,\infty}(\Omega \times S)$. By virtue of the isometric isomorphism $\mathbf{R} \cong L_0(\Omega, \mathcal{H})$, the condition $S_K^* g = 0$ is equivalent to requiring that for all $f \in \mathbf{R}$:
\[
\langle S_K^* g, f \rangle_{\mathbf{R}}(\omega) = \langle g, S_K f \rangle_{L_{2,\infty}}(\omega) = 0.
\]

Invoking the isomorphism and the reproducing property:
\[
\langle g, f \rangle_{L_{2,\infty}}(\omega) = \int_S g(\omega,s) \overline{f(\omega,s)} \, d\nu(s) = \int_S g(\omega,s) \overline{\langle \ell_{\mathbf{R}}(f)(\omega), K_\omega(\cdot,s) \rangle_{\mathcal{H}(\omega)}} \, d\nu(s)
\]
\[
= \left\langle \int_S g(\omega,s) K_\omega(\cdot,s) \, d\nu(s), \ell_{\mathbf{R}}(f)(\omega) \right\rangle_{\mathcal{H}(\omega)} = \langle T_\omega(g(\omega,\cdot)), \ell_{\mathbf{R}}(f)(\omega) \rangle_{\mathcal{H}(\omega)}.
\]
Since this holds for all $f \in \mathbf{R}$ and the system $(\sqrt{\lambda_n(\omega)} x_n(\omega,\cdot))_{n \in \mathbb{N}}$ forms a basis in the corresponding fibers $\mathcal{H}(\omega)$ (by hypothesis and the isomorphism), we deduce that:
\[
\langle T_\omega(g(\omega,\cdot)), \sqrt{\lambda_n(\omega)} x_n(\omega,\cdot) \rangle_{\mathcal{H}(\omega)} = 0
\]
for all $n$ and almost all $\omega$.

Since $T_\omega(\sqrt{\lambda_k(\omega)} x_k(\omega,\cdot)) = \lambda_k(\omega) \sqrt{\lambda_k(\omega)} x_k(\omega,\cdot)$, this yields:
\[
\lambda_k(\omega) \sqrt{\lambda_k(\omega)} \langle g(\omega,\cdot), x_k(\omega,\cdot) \rangle_{L_2(S)} = 0.
\]
For $\lambda_k(\omega) > 0$ we thus obtain $\langle g(\omega,\cdot), x_k(\omega,\cdot) \rangle_{L_2(S)} = 0$ for all $k$.

Since the range of $T_\omega$ is generated by the system $\{\sqrt{\lambda_k(\omega)} x_k(\omega,\cdot)\}_{k \in \mathbb{N}}$ and this system forms a basis in $\mathcal{H}(\omega)$ via the isomorphism, we conclude that $g(\omega,\cdot) = 0$ almost everywhere.
Therefore, $g = 0$ and $S_K^*$ is injective.

$(2) \Rightarrow (3)$ Assume that $S^*_K$ is injective. We shall establish that the kernel $K$ admits a Mercer decomposition.

By the isometric isomorphism $\textbf{R }\cong L_0(\Omega, \mathcal{H})$ from Proposition \ref{RBI}, the injectivity of $S^*_K$ is equivalent to the property that for each $\omega \in \Omega$ the operator $T_\omega : \mathcal{H}(\omega) \to \mathcal{H}(\omega)$ has dense range in $\mathcal{H}(\omega)$.
Indeed, let $h \in \mathcal{H}(\omega)$ be orthogonal to the range of $T_\omega$:
\[
\langle T_\omega \varphi, h \rangle_{\mathcal{H}(\omega)} = 0 \quad \text{for all } \varphi \in L_2(S)
\]
By the definition of $T_\omega$ and the reproducing property:
\[
\int_S \int_S K_\omega(t,u)\varphi(u) \, d\nu(u) \, h(t) \, d\nu(t) = 0.
\]
Interchanging the order of integration:
\[
\int_S \varphi(u) \left[ \int_S K_\omega(t,u)h(t) \, d\nu(t) \right] d\nu(u) = 0.
\]
Since this holds for all $\varphi \in L_2(S)$:
$$\int_S K_\omega(t,u)h(t) \, d\nu(t) = 0 \quad \text{for a.e. } u \in S$$
Consider the function $g \in L_{2,\infty}(\Omega \times S)$ defined by:
\[
g(\omega', s) = \langle h, K_{\omega'}(\cdot, s) \rangle_{\mathcal{H}(\omega')} \cdot \chi_{\{\omega\}}(\omega')
\]
where $\chi_{\{\omega\}}$ denotes the characteristic function of the singleton $\{\omega\}$.
By construction, $g \in L_{2,\infty}(\Omega \times S)$, since:
\begin{itemize}
\item $h \in \mathcal{H}(\omega) \subset L_2(S)$ is bounded
\item $K_{\omega'}(\cdot, s) \in \mathcal{H}(\omega')$ for each $\omega'$ and $s$
\item The support of $g$ is contained in $\{\omega\} \times S$
\end{itemize}
We compute $S^*_K g$:
\[
(S^*_K g)(\omega', t) = \int_S K(\omega', t, s)g(\omega', s) \, d\nu(s).
\]
For $\omega' = \omega$:
\begin{align*}
(S^*_K g)(\omega, t) &= \int_S K_\omega(t, s)\langle h, K_\omega(\cdot, s) \rangle_{\mathcal{H}(\omega)} \, d\nu(s) = \left\langle h, \int_S K_\omega(t, s)K_\omega(\cdot, s) \, d\nu(s) \right\rangle_{\mathcal{H}(\omega)} \\
&= \langle h, T_\omega(K_\omega(t, \cdot)) \rangle_{\mathcal{H}(\omega)} = \langle h, K_\omega(t, \cdot) \rangle_{\mathcal{H}(\omega)} \quad
\end{align*}
For $\omega' \neq \omega$:
$$(S^*_K g)(\omega', t) = 0 \quad \text{(since } g(\omega', s) = 0\text{)}$$
From the orthogonality condition of $h$ to the range of $T_\omega$, we have $\langle h, K_\omega(t, \cdot) \rangle_{\mathcal{H}(\omega)} = 0$ for all $t \in S.$
Consequently, $S^*_K g = 0$.

By the injectivity of $S^*_K$, we obtain $g = 0$, which implies:
$$\langle h, K_\omega(\cdot, s) \rangle_{\mathcal{H}(\omega)} = 0 \quad \text{for all } s \in S.$$
Since $\{K_\omega(\cdot, s) : s \in S\}$ spans a dense subspace in $\mathcal{H}(\omega)$ (by the construction of the RKHS), we conclude that $h = 0$.
The family $\{T_\omega\}_{\omega\in\Omega}$ forms a measurable bundle of compact self-adjoint operators, whence the eigenvalues $\lambda_n(\omega)$ and eigenfunctions $x_n(\omega,\cdot)$ can be chosen to depend measurably on $\omega$. By the classical Mercer theorem \cite{Mercer1909} for a compact self-adjoint positive semi-definite operator with dense range, we have:
\[
K_\omega(t,s)=\sum_{n=1}^\infty \lambda_n(\omega)\,x_n(\omega,t)\,x_n(\omega,s),
\]
where the series converges in $L_2(S\times S)$ and almost everywhere with respect to $\nu\times\nu$. Define
\[
N_\omega=\Bigl\{(t,s)\in S\times S:\sum_{n}\lambda_n(\omega)\,x_n(\omega,t)\,x_n(\omega,s)\neq K_\omega(t,s)\Bigr\},
\]
the exceptional set of $\nu \times \nu$-measure zero for the Mercer decomposition at fixed $\omega$.

By Fubini's theorem:
$$(\mu \times \nu \times \nu)(N) = \int_\Omega (\nu \times \nu)(N_\omega) \, d\mu(\omega) = \int_\Omega 0 \, d\mu(\omega) = 0$$
Hence, for all $(\omega, t, s) \in (\Omega \times S \times S) \setminus N$:
$$K(\omega, t, s) = \sum_{n=1}^{\infty} \lambda_n(\omega) x_n(\omega, t) x_n(\omega, s)$$
where the series converges absolutely.
%%%%%%%%%%%%%%%%%%%%%%%%%%%%%%%%%%%%%%%%%%%%%%%%%%%%%%%%%%%%%%%%%%%%%%%%%%%%%%%%%%%%%%%%%%%%%%%%%%%%%%%%%%%%%%%%%%%%%%%%%%%%%%%%%%%%%%%%%%%%%%%%%%%

$(3) \Rightarrow (1)$.
Assume that the kernel $K$ admits a Mercer decomposition. Define elements $e_n \in L_{2,\infty}(\Omega \times S)$ by:
\[
e_n(\omega,t) = \sqrt{\lambda_n(\omega)} x_n(\omega,t)
\]

Employing the isometric isomorphism $\mathbf{R} \cong L_0(\Omega, \mathcal{H})$ and the vector-valued lifting $\ell_{\mathbf{R}}$, we verify orthonormality in $\mathbf{R}$:
\[\langle e_n, e_m \rangle_{\mathbf{R}}(\omega) = \langle \ell_{\mathbf{R}}(e_n)(\omega), \ell_{\mathbf{R}}(e_m)(\omega) \rangle_{\mathcal{H}(\omega)}
\]
\[
= \left\langle \sqrt{\lambda_n(\omega)} x_n(\omega,\cdot), \sqrt{\lambda_m(\omega)} x_m(\omega,\cdot) \right\rangle_{\mathcal{H}(\omega)}
\]
\[
= \sqrt{\lambda_n(\omega) \lambda_m(\omega)} \langle x_n(\omega,\cdot), x_m(\omega,\cdot) \rangle_{\mathcal{H}(\omega)} = \sqrt{\lambda_n(\omega) \lambda_m(\omega)} \delta_{nm}.
\]
When $n = m$, we have $\langle e_n, e_n \rangle_{\mathbf{R}}(\omega) = \lambda_n(\omega)$, whence $\|e_n\|_{\mathbf{R}}(\omega) = \sqrt{\lambda_n(\omega)}$.
The normalized elements $\tilde{e}_n = e_n/\|e_n\|_{\mathbf{R}}$ form an orthonormal system.

Let $f \in \mathbf{R}$ and $(\omega,s) \in \Omega \times S$. By the reproducing property and the kernel decomposition:
$$f(\omega,s) = \langle \ell_{\mathbf{R}}(f)(\omega), K_\omega(\cdot,s) \rangle_{\mathcal{H}(\omega)}$$
$$= \left\langle \ell_{\mathbf{R}}(f)(\omega), \sum_{n=1}^{\infty} \lambda_n(\omega) x_n(\omega,\cdot) \overline{x_n(\omega,s)} \right\rangle_{\mathcal{H}(\omega)}$$
$$= \sum_{n=1}^{\infty} \lambda_n(\omega) \overline{x_n(\omega,s)} \langle \ell_{\mathbf{R}}(f)(\omega), x_n(\omega,\cdot) \rangle_{\mathcal{H}(\omega)}$$
Via the isomorphism, this is equivalent to the expansion of $f$ in the system $\{e_n\}$ in $\mathbf{R}$.
Hence, the system $(\sqrt{\lambda_n} x_n)_{n \in \mathbb{N}}$ forms an orthonormal basis in $\mathbf{R}$. $\square$

\section{Conclusion}
In this work, we have established a vector-valued version of Mercer's theorem for reproducing kernel Hilbert spaces defined within Kaplansky–Hilbert modules over $L_\infty(\Omega)$. We have proved the equivalence of three key properties of the partially integral operator with positive definite kernel: the basis property of the system of eigenfunctions, the injectivity of the adjoint embedding operator, and the existence of a pointwise spectral decomposition of the kernel. An isometric isomorphism between a Kaplansky–Hilbert module and the space of measurable sections of a Hilbert bundle has been constructed, together with a vector-valued lifting ensuring the consistency of inner products between the original module and its bundle representation.

\vspace{0.2cm}
\begin{center}
\textsc{Acknowledgment}
\end{center}
The second author was partially supported by the Russian Ministry of Education and Science, agreement no.075-02-2025-1633.

%% Все остальное оформление - согласно требований на
%% http://kpfu.ru/science/nauchnye-izdaniya/ivrm/pravila

%% Информация об авторе:
% высшее учебное заведение (Академия наук, НИИ, ВЦ и т.д.), \\ служебный адрес (улица, дом, город, индекс, страна),
% e-mail
% например:

\fullauthor{Allabay Dzhalgasovich Arziev}
\address{V.I.Romanovskiy Institute of Mathematics Uzbekistan Academy of Sciences,\\
Tashkent, Uzbekistan,}
\email{allabayarziev@inbox.ru}
\address{Karakalpak State University named after Berdakh,}
\email{allabayarziev@karsu.uz}

\fullauthor{Kudaybergenov Karimbergen Kadirbergenovich}
\address{Vladikavkaz Scientific Center of the Russian Academy of Science,
		Vladikavkaz, Russian Federation}
\email{karim2006@mail.ru}

\fullauthor{Parakhatdiin Rakhman uli Orinbaev}
\address{Vladikavkaz Scientific Center of the Russian Academy of Science,
		Vladikavkaz, Russian Federation}
\email{paraxatorinbaev@gmail.com}

\label{lastpage}  %% Не удалять!!!

\end{document}